\documentclass[12pt]{amsart}
\usepackage{amsfonts, amsmath, amsthm}

\usepackage{graphicx}
\usepackage{epstopdf}
\usepackage{psfrag}

\newtheorem{pro}{Proposition}[section]
\newtheorem{thm}[pro]{Theorem}
\newtheorem{lem}[pro]{Lemma}

\newtheorem{cor}[pro]{Corollary}

\theoremstyle{definition}
\newtheorem{dfn}[pro]{Definition}

\theoremstyle{remark}

\newcommand{\VV}{\mathcal V}
\newcommand{\WW}{\mathcal W}
\newcommand{\CC}{\mathcal C}

\newcommand{\gen}{\mbox{\rm genus}}

\newcommand{\bdy}{\partial}

\newcommand{\thick}[1]{{\rm Thick}(#1)}
\newcommand{\thin}[1]{{\rm Thin}(#1)}

\newcommand{\amlg}[1]{\mathcal A(#1)}

\newcommand{\CV}{\mathcal V}
\newcommand{\CW}{\mathcal W}

\title{Heegaard splittings of sufficiently complicated 3-manifolds I: Stabilization} 
\date{\today}
\address{Pitzer College}
\email{bachman@pitzer.edu}
\author{David Bachman}
\begin{document}

\begin{abstract}
We construct families of pairs of Heegaard splittings that must be stabilized several times to become equivalent. The first such pair differs only by their orientation. These are genus $n$ splittings of a closed 3-manifold that must be stabilized at least $n-2$ times to become equivalent. The second is a pair of genus $n$ splittings of a manifold with toroidal boundary that must be stabilized at least $n-4$ times to become equivalent. The last example is a pair of genus $n$ splittings of a closed 3-manifold that must be stabilized at least $\frac{1}{2}n -3$ times to become equivalent, regardless of their orientations. All of these examples are splittings of manifolds that are obtained from simpler manifolds by gluing along incompressible surfaces via ``sufficiently complicated" maps. 
\end{abstract}

\maketitle

\markright{SPLITTINGS OF SUFFICIENTLY COMPLICATED 3-MANIFOLDS I}

\section{Introduction.}

Given a Heegaard surface $H$ in a 3-manifold, $M$, one can {\it stabilize} to obtain a splitting of higher genus by taking the connected sum of $H$ with the genus one splitting of $S^3$. Suppose $H_1$ and $H_2$ are Heegaard splittings of $M$, where $\rm{genus}(H_1) \ge \rm{genus}(H_2)$. It is a classical result of Reidemeister \cite{reidemeister} and Singer \cite{singer} from 1933 that as long as $H_1$ and $H_2$ induce the same partition of the components of $\partial M$, stabilizing $H_1$ some number of times produces a stabilization of $H_2$. Just one stabilization was proved to always be sufficient in large classes of 3-manifolds, including Seifert fibered spaces \cite{schultens:96}, genus two 3-manifolds \cite{rs:99}, and most graph manifolds \cite{Derby-Talbot2006} (see also \cite{sedgwick:97}). The lack of examples to the contrary led to ``The Stabilization Conjecture": Any pair of Heegaard splittings requires at most one stabilization to become equivalent. (See Conjecture 7.4 in \cite{ScharlemannSurvey}.)

The purpose of this paper is to produce several families of counter-examples to the Stabilization Conjecture. This work was announced in December of 2007 at a Workshop on {\it Triangulations, Heegaard Splittings, and Hyperbolic Geometry}, at the American Institute of Mathematics. At the same conference another family of counter-examples to the Stabilization Conjecture was announced by Hass, Thompson, and Thurston, and their preprint has since appeared on the arXiv \cite{HTT}. Their proof uses mainly geometric techniques. Several months later Johnson posted a preprint on the arXiv \cite{johnson} containing similar results that was motivated by this work, but is completely combinatorial. The proofs presented here are quite different than either of these. 

Here we construct three families of counter-examples. These are described by the following theorems:

\medskip

\noindent {\bf Theorem \ref{t:FlipCounterExample}.} {\it For each $n \ge 4$ there is a closed, orientable 3-manifold that has a genus $n$ Heegaard splitting which must be stabilized at least $n-2$ times to become equivalent to the splitting obtained from it by reversing its orientation.}

\medskip 

\noindent {\bf Theorem \ref{t:TorusBoundaryCounterExamples}.} {\it For each $n \ge 5$ there is an orientable 3-manifold whose boundary is a torus, that has two genus $n$ Heegaard splittings which must be stabilized at least $n-4$ times to become equivalent.}

\medskip

\noindent {\bf Theorem \ref{t:ClosedCounterExamples}.} {\it For each $n \ge 8$ there is a closed, orientable 3-manifold that has a pair of genus $n$ Heegaard splittings which must be stabilized at least $\frac{1}{2}n -3$ times to become equivalent (regardless of their orientations).}

\medskip

The key to the constructions of the counter-examples given in \cite{HTT}, \cite{johnson} and \cite{johnson2}  is to use Heegaard splittings formed by gluing together two handlebodies by a very complicated homeomorphism. Such splittings have very high Hempel distance \cite{hempel:01}. In contrast, the examples constructed here are splittings of manifolds that are constructed from two or more component manifolds by gluing along incompressible surfaces via a very complicated map. The splittings themselves come from amalgamations of splittings of the component manifolds, and so have Hempel distance at most one.

This paper is organized as follows. In Sections \ref{s:FirstDefSection} through \ref{s:LastDefSection} we mostly review the definitions and results given in \cite{gordon}. These include {\it critical surfaces}, {\it Generalized Heegaard splittings} (GHSs), and {\it Sequences of GHSs} (SOGs). In Section \ref{s:BarrierSurfaces} we review the main result from \cite{barrier}, which says that complicated amalgamations act as barriers to low genus incompressible, strongly irreducible, and critical surfaces. 
In Section \ref{s:CounterExamples} we use all of this machinery to construct our counter-examples to the Stabilization Conjecture. In the sequel \cite{AmalgamationResults} we use the machinery presented in Sections \ref{s:FirstDefSection} through \ref{s:BarrierSurfaces} to show that ``sufficiently complicated" amalgamations of unstabilized, boundary-unstabilized Heegaard splittings are also unstabilized.

\section{Incompressible, Strongly Irreducible, and Critical surfaces}
\label{s:FirstDefSection}

In this section we recall the definitions of various classes of topologically interesting surfaces. The first are the {\it incompressible} surfaces of Haken \cite{haken:68}, which have played a central role in 3-manifold topology. The second class are the {\it strongly irreducible} surfaces of Casson and Gordon \cite{cg:87}. These surfaces have become important in answering a wide variety of questions relating to the Heegaard genus of 3-manifolds. The third class are the {\it critical surfaces} of \cite{crit} and \cite{gordon}. 

In \cite{TopIndexI} we show that all three of these classes are special cases of {\it topologically minimal} surfaces. Such surfaces are the topological analogue of geometrically minimal surfaces. We will say more about this in Section \ref{s:BarrierSurfaces}. 

\begin{dfn}
\label{d:essential}
Let $F$ be a properly embedded surface in $M$. Let $\gamma$ be a loop in $F$. $\gamma$ is {\it essential} on $F$ if it is a loop that does not bound a disk in $F$.  A {\it compression} for $F$ is a disk, $D$, such that $D \cap F=\bdy D$ is essential on $F$.
\end{dfn}

\begin{dfn}
Let $F$ be a properly embedded surface in $M$. The surface $F$ is {\it compressible} if there is a compression for it. Otherwise it is {\it incompressible}.
\end{dfn}

\begin{dfn}
Let $H$ be a separating, properly embedded surface in $M$. Let $V$ and $W$ be compressions on opposite sides of $H$. Then we say $(V,W)$ is a {\it weak reducing pair} for $H$ if $V \cap W=\emptyset$. 
\end{dfn}

\begin{dfn}
\label{d:TTStrongIrreducibility}
Let $H$ be a separating, properly embedded surface in $M$ which is not a torus. Then we say $H$ is {\it strongly irreducible} if there are compressions on opposite sides of $H$, but no weak reducing pairs. 
\end{dfn}

\begin{dfn}
\label{d:critical}
Let $H$ be a properly embedded, separating surface in $M$. The surface $H$ is {\it critical} if the compressions for $H$ can be partitioned into sets $C_0$ and $C_1$ such that:
\begin{enumerate}
	\item For each $i=0,1$ there is at least one pair of disks $V_i, W_i \in C_i$ such that $(V_i,W_i)$ is a weak reducing pair. 
	\item If $V \in C_0$ and $W \in C_1$ then $(V,W)$ is not a weak reducing pair.
\end{enumerate}
\end{dfn}

\section{Generalized Heegaard Splittings}

In this section we define {\it Heegaard splittings} and {\it Generalized Heegaard Splitting}. The latter structures were first introduced by Scharlemann and Thompson \cite{st:94} as a way of keeping track of handle structures. The definition we give here is more consistent with the usage in \cite{gordon}.

\begin{dfn}
A {\it compression body} $\mathcal C$ is a manifold formed in one of the following two ways:
	\begin{enumerate}
		\item Starting with a 0-handle, attach some number of 1-handles. In this case we say $\bdy _- \mathcal C=\emptyset$ and $\bdy _+ \mathcal C=\bdy \mathcal C$. 
		\item Start with some (possibly disconnected) surface $F$ such that each component has positive genus. Form the product $F \times I$. Then attach some number of 1-handles to $F \times \{1\}$. We say $\bdy _- \mathcal C= F \times \{0\}$ and $\bdy _+ \mathcal C$ is the rest of $\bdy \mathcal C$. 
	\end{enumerate}
\end{dfn}

\begin{dfn}
\label{d:Heegaard}
Let $H$ be a properly embedded, transversally oriented surface in a 3-manifold $M$, and suppose $H$ separates $M$ into $\VV$ and $\WW$. If $\VV$ and $\WW$ are compression bodies and $\VV \cap \WW=\partial _+ \VV=\partial _+ \WW=H$, then we say $H$ is a {\it Heegaard surface} in $M$.
\end{dfn}

\begin{dfn}
The transverse orientation on the Heegaard surface $H$ in the previous definition is given by a choice of normal vector. If this vector points into $\VV$, then we say any subset of $\VV$ is {\it above} $H$ and any subset of $\WW$ is {\it below} $H$. 
\end{dfn}

\begin{dfn}
Suppose $H$ is a Heegaard splitting of a manifold $M$ with non-empty boundary. Let $F$ denote a component of $\bdy M$. Then the surface $H'$ obtained from $H$ by attaching a copy of $F$ to it by an unknotted tube is also a Heegaard surface in $M$. We say $H'$ was obtained from $H$ by a {\it boundary-stabilization along $F$}. The reverse operation is called a {\it boundary-destabilization} along $F$. 
\end{dfn}

\begin{dfn}
\label{d:GHS}
A {\it generalized Heegaard splitting (GHS)} $H$ of a 3-manifold $M$ is a pair of sets of transversally oriented, connected, properly embedded surfaces,  $\thick{H}$ and $\thin{H}$ (called the {\it thick levels} and {\it thin levels}, respectively), which satisfy the following conditions. 
	\begin{enumerate}
		\item Each component $M'$ of $M \setminus \thin{H}$ meets a unique element $H_+$ of $\thick{H}$. The surface $H_+$ is a Heegaard surface in $\overline{M'}$ dividing $\overline{M'}$ into compression bodies $\VV$ and $\WW$. Each component of $\bdy _- \VV$ and $\bdy _- \WW$ is an element of $\thin{H}$. Henceforth we will denote the closure of the component of $M \setminus \thin{H}$ that contains an element $H_+ \in \thick{H}$ as $M(H_+)$. 
		\item Suppose $H_-\in \thin{H}$. Let $M(H_+)$ and $M(H_+')$ be the submanifolds on each side of $H_-$. Then $H_-$ is below $H_+$ in $M(H_+)$ if and only if it is above $H_+'$ in $M(H_+')$.
		\item The term ``above" extends to a partial ordering on the elements of $\thin{H}$ defined as follows. If $H_-$ and $H'_-$ are subsets of $\bdy M(H_+)$, where $H_-$ is above $H_+$ in $M(H_+)$ and $H_-'$ is below $H_+$ in $M(H_+)$, then $H_-$ is above $H_-'$ in $M$.
	\end{enumerate}
\end{dfn}

\begin{dfn}
Suppose $H$ is a GHS of an irreducible 3-manifold $M$. Then $H$ is {\it strongly irreducible} if each element $H_+ \in \thick{H}$ is strongly irreducible in $M(H_+)$. The GHS $H$ is {\it critical} if each element $H_+ \in \thick{H}$ but exactly one is strongly irreducible in $M(H_+)$, and the remaining element is critical in $M(H_+)$. 
\end{dfn}

The strongly irreducible case of the following result is due to Scharlemann and Thompson \cite{st:94}. The proof in the critical case is similar. 

\begin{thm}
\label{t:IncompressibleThinLevels}{\rm (\cite{gordon}, Lemma 4.6)} 
Suppose $H$ is a strongly irreducible or critical GHS of an irreducible 3-manifold $M$. Then each thin level of $H$ is incompressible.
\end{thm}

\section{Reducing GHSs}

\begin{dfn}
Let $H$ be an embedded surface in $M$.  Let $D$ be a compression for $H$. Let $\VV$ denote the closure of the component of $M \setminus H$ that contains $D$. (If $H$ is non-separating then $\VV$ is the manifold obtained from $M$ by cutting open along $H$.) Let $N$ denote a regular neighborhood of $D$ in $\VV$. To {\it surger} or {\it compress} $H$ along $D$ is to remove $N \cap H$ from $H$ and replace it with the frontier of $N$ in $\VV$. We denote the resulting surface by $H/D$. 
\end{dfn}

It is not difficult to find a complexity for surfaces which decreases under compression. Incompressible surfaces then represent ``local minima" with respect to this complexity. We now present an operation that one can perform on GHSs that also reduces some complexity (see Lemma 5.14 of \cite{gordon}).  Strongly irreducible GHSs will then represent ``local minima" with respect to such a complexity. This operation is called {\it weak reduction}.

\begin{dfn}
Let $H$ be a properly embedded surface in $M$. If $(D,E)$ is a weak reducing pair for $H$, then we let $H/DE$ denote the result of simultaneous surgery along $D$ and $E$.
\end{dfn}

\begin{dfn}
\label{d:PreWeakReduction}
Let $M$ be a compact, connected, orientable 3-manifold. Let $G$ be a GHS. Let $(D,E)$ be a weak reducing pair for some  $G_+ \in \thick{G}$. Define 
	\[T(H)=\thick{G} -\{G_+\} \cup \{G_+ / D, G_+ / E\},\ \mbox{and}\]
	\[t(H) =\thin{G} \cup \{G_+/DE\}.\]

A new GHS $H=\{\thick{H},\thin{H}\}$ is then obtained from $\{T(H), t(H)\}$ by successively removing the following:
	\begin{enumerate}
		\item Any sphere element $S$ of $T(H)$ or $t(H)$ that is inessential, along with any elements of $t(H)$ and $T(H)$ that lie in the ball that it (co)bounds. 
		\item Any element $S$ of $T(H)$ or $t(H)$ that is $\bdy$-parallel, along with any elements of $t(H)$ and $T(H)$ that lie between $S$ and $\bdy M$. 
		\item Any elements $H_+ \in T(H)$ and $H_- \in t(H)$, where $H_+$ and $H_-$ cobound a submanifold $P$ of $M$, such that $P$ is a product, $P \cap T(H)=H_+$, and $P \cap t(H)=H_-$. 
	\end{enumerate}

We say the GHS $H$ is obtained from $G$ by {\it weak reduction} along $(D,E)$.
\end{dfn}

The first step in weak reduction is illustrated in Figure \ref{f:WeakReduction}. 

        \begin{figure}[htbp]
        \psfrag{1}{$G_+/D$}
        \psfrag{2}{$G_+/E$}
        \psfrag{3}{$G_+/DE$}
        \psfrag{G}{$G_+$}
        \psfrag{E}{$E$}
        \psfrag{D}{$D$}
        \vspace{0 in}
        \begin{center}
       \includegraphics[width=3.5 in]{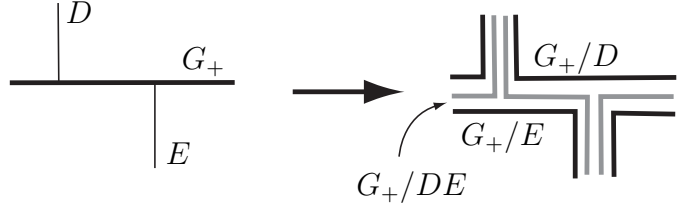}
       \caption{The first step in weak reduction.}
        \label{f:WeakReduction}
        \end{center}
        \end{figure}

\begin{dfn}
\label{d:destabilization}
The weak reduction of a GHS given by the weak reducing pair $(D,E)$ for the thick level $G_+$ is called a {\it destabilization} if $G_+/DE$ contains a sphere.
\end{dfn} 

In the next section we give a coarse measure of complexity for GHSs called {\it genus}. Destabilizations are precisely those weak reductions that reduce genus.

\section{Amalgamations}

Let $H$ be a GHS of a connected 3-manifold $M$. In this section we use $H$ to produce a complex that is the spine of a Heegaard splitting of $M$. We call this splitting the {\it amalgamation} of $H$. Most of this material is reproduced from \cite{gordon}. First, we must introduce some new notation. 

\begin{dfn}
Let $H$ be a Heegaard surface in $M$. Let $\Sigma$ denote a properly embedded graph in $M$. Let $(\bdy M) '$ denote the union of the boundary components of $M$ that meet $\Sigma$. Then we say $(\bdy M)' \cup \Sigma$ is a {\it spine} of $H$ if the frontier of a neighborhood of $(\bdy M)' \cup \Sigma$ is isotopic to $H$.
\end{dfn}

Suppose $H$ is a GHS of $M$ and $H_{+}\in \thick{H}$. Recall that $H_{+}$ is transversely oriented, so that we may consistently talk about those points of $M(H_+)$ that are ``above" $H_+$ and those points that are ``below." The surface $H_+$ divides $M(H_+)$ into two compression bodies. Henceforth we will denote these compression bodies as $\CV(H_+)$ and $\CW(H_+)$, where $\CV(H_+)$ is below $H_+$ and $\CW(H_+)$ is above. When we wish to make reference to an arbitrary compression body which lies above or below some thick level we will use the notation $\CV$ and $\CW$. Define $\partial _- M(H_+)$ to be $\partial _- \CV(H_+)$ and $\partial _+ M(H_+)$ to be $\partial _- \CW (H_+)$. That is, $\partial _- M(H_+)$ and $\partial _+ M(H_+)$ are the boundary components of $M(H_+)$ that are below and above $H_+$, respectively. If $N$ is a union of manifolds of the form $M(H_i)$ for some set of thick levels $\{H_i\} \subset \thick{H}$ then we let $\bdy _{\pm} N$ denote the union of those boundary components of $N$ that are components of $\bdy _{\pm} M(H_i)$, for some $i$. 

We now define a sequence of manifolds $\{M_i\}$ where
\[M_0 \subset M_1 \subset ... \subset M_n=M.\]
The submanifold $M_0$ is defined to be the disjoint union of all manifolds of the form $M(H_+)$, such that  $\partial _- M(H_+) \subset \bdy M$. The fact that the thin levels of $H$ are partially ordered guarantees $M_0 \ne \emptyset$. Now, for each $i$ we define $M_i$ to be the union of $M_{i-1}$ and all manifolds $M(H_+)$ such that $\partial _- M(H_+) \subset \partial M_{i-1} \cup \bdy M$. Again, it follows from the partial ordering of thin levels that for some $i$ the manifold $M_i=M$. 

We now define a sequence of complexes $\Sigma_i$ in $M$. The final element of this sequence will be a complex $\Sigma$. This complex will be a spine of the desired Heegaard surface. The intersection of $\Sigma$ with some $M(H_+)$ is depicted in Figure \ref{f:Amalgam}.

Each $\CV \subset M_0$ is a compression-body. Choose a spine of each, and let $\Sigma'_0$ denote the union of these spines. The complement of $\Sigma'_0$ in $M_0$ is a (disconnected) compression body, homeomorphic to the union of the compression bodies $\CW \subset M_0$. Now let $\Sigma_0$ be the union of $\Sigma'_0$ and one vertical arc for each component $H_-$ of $\partial _+ M_0$,  connecting $H_-$ to $\Sigma' _0$. 

We now assume $\Sigma _{i-1}$ has been constructed and we construct $\Sigma_i$. Let $M_i'=\overline{M_i - M_{i-1}}$. For each compression body $\CV \subset M_i'$ choose a set of arcs $\Gamma \subset \CV$ such that $\partial \Gamma \subset \Sigma _{i-1} \cap \partial M_{i-1}$, and such that the complement of $\Gamma$ in $\CV$ is a product. Let $\Sigma' _i$ be the union of $\Sigma _{i-1}$ with all such arcs $\Gamma$, and all components of $\bdy _- \CV$ that are contained in $\bdy M$. Now let $\Sigma_i$ be the union of $\Sigma'_i$ and one vertical arc for each component $H_-$ of $\partial _+ M_i$, connecting $H_-$ to $\Sigma' _i$. 

        \begin{figure}[htbp]
        \psfrag{h}{$H_+$}
        \psfrag{W}{$\CV(H_+)$}
        \psfrag{w}{$\CW(H_+)$}
        \psfrag{S}{$\Sigma$}
        \vspace{0 in}
        \begin{center}
       \includegraphics[width=3 in]{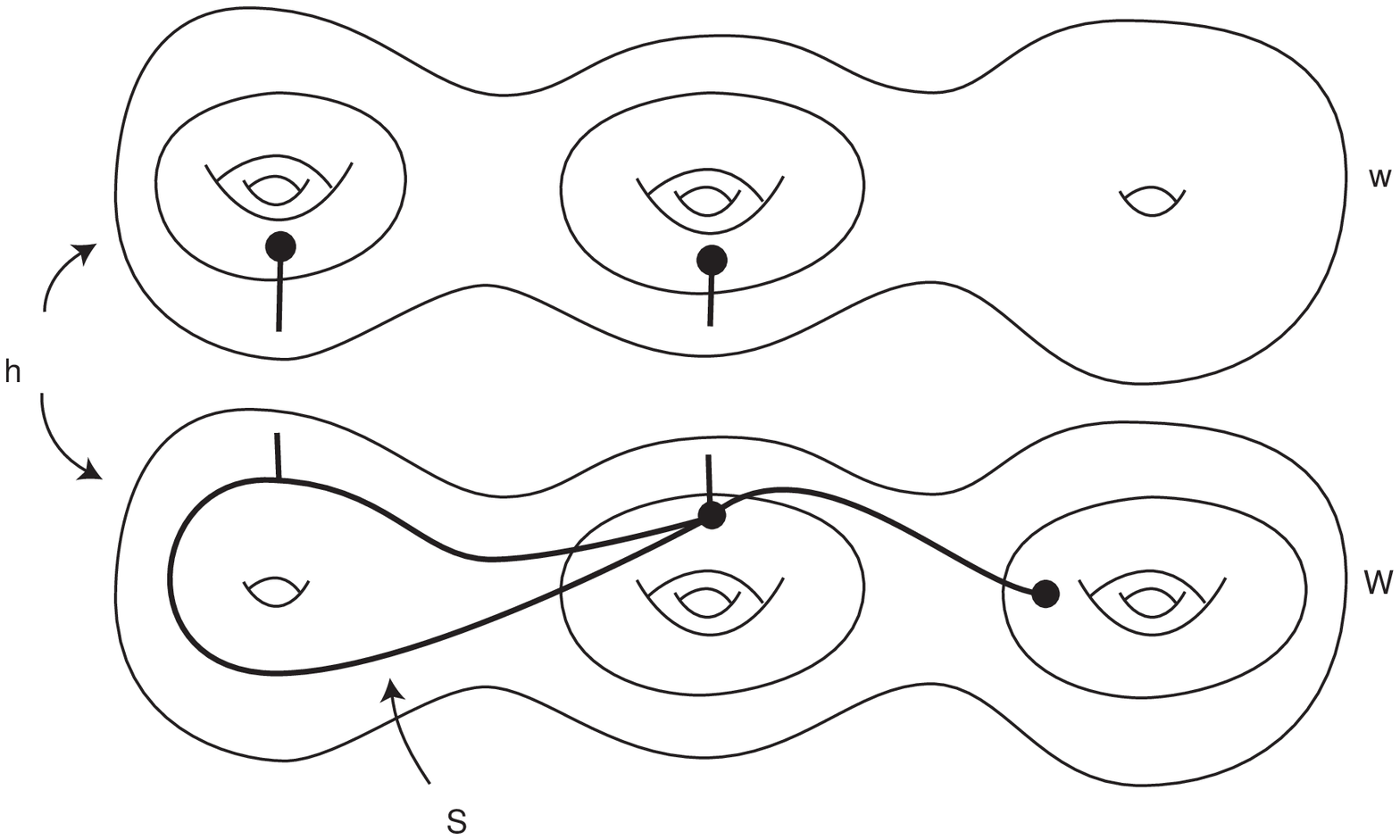}
       \caption{The intersection of $\Sigma$ with $\CV(H_+)$ and $\CW(H_+)$.}
        \label{f:Amalgam}
        \end{center}
        \end{figure}

\begin{lem}
{\rm (\cite{gordon}, Lemma 7.2)}  If $H$ is a GHS of $M$ then the complex $\Sigma$ defined above is the spine of a Heegaard splitting of $M$.
\end{lem}

\begin{dfn}
\label{d:amalgam}
Let $H$ be a GHS and $\Sigma$ be the complex in $M$ defined above. The Heegaard splitting that $\Sigma$ is a spine of is called the {\it amalgamation} of $H$ and will be denoted $\amlg{H}$.
\end{dfn}

Note that although the construction of the complex $\Sigma$ involved some choices, its neighborhood is uniquely defined up to isotopy at each stage. Hence, the amalgamation of a GHS is well defined, up to isotopy. 

For the next lemma, recall the definition of {\it destabilization}, given in Definition \ref{d:destabilization}.

\begin{lem}
\label{l:AmalgGenus}
{\rm (\cite{gordon}, Corollary 7.5)}  Suppose $M$ is irreducible, $H$ is a GHS of $M$ and $G$ is obtained from $H$ by a weak reduction which is not a destabilization. Then $\amlg{H}$ is isotopic to $\amlg{G}$.
\end{lem}

It follows that if a GHS $G$ is obtained from a GHS $H$ by a weak reduction or a destabilization then the genus of $\amlg{G}$ is at most the genus of $\amlg{H}$.

\begin{dfn}
The {\it genus} of a GHS is the genus of its amalgamation. 
\end{dfn}

\begin{dfn}
Suppose $H$ is a GHS of $M$. Let $N$ denote a submanifold of $M$ bounded by elements of $\thin{H}$. Then we may define a GHS $H(N)$ of $N$. The thick and thin levels of $H(N)$ are the thick and thin levels of $H$ that lie in $N$. 
\end{dfn}


\begin{lem}
\label{l:GenusSum}
Suppose $H$ is a GHS of $M$, $F$ is an arbitrary subset of $\thin{H}$ in the interior of $M$, and $\{M_i\}_{i=1}^n$ are the closures of the components of $M \setminus F$. Then
\[\gen(H)=\sum \limits _{i=1} ^n \gen (H(M_i)) -\gen(F) +|F|-n+1.\]
\end{lem}

\begin{proof}
The proof is by induction on $|F|$. Suppose first $F$ is connected, so that $|F|=1$. There are then two cases, depending on whether or not $F$ separates $M$. 

We first deal with the case where $F$ separates  $M$ into $M_1$ and $M_2$. In this case $|F|-n+1=1-2+1=0$, so we need to establish \[\gen(H)=\gen(H(M_1))+\gen(H(M_2))-\gen(F).\]

Let $\Sigma (M_1)$ and $\Sigma (M_2)$ denote spines of $\amlg{H(M_1)}$ and $\amlg{H(M_2)}$. Then $\Sigma (M_1)$ is the union of a properly embedded graph $\Sigma (M_1)' \subset M_1$ and $\bdy_- M_1$. If $M_1$ is above $M_2$ in $M$ then $F$ is a component of $\bdy _- M_1$. Let $(\bdy _-M_1)'=\bdy _-M_1 \setminus F$. 

To form the spine of $\amlg{H}$ we attach a vertical arc from $\Sigma (M_2)$ to  $\Sigma (M_1)' \cup (\bdy _- M_1)'$, through the compression body in $M_2$ that is incident to $F$. Hence, the graph part of the spine of $H$ comes from the graph parts of $\Sigma (M_1)$ and $\Sigma(M_2)$, together with an arc. The surface part only comes from the surface part of $\Sigma(M_2)$ and the surface parts of $\Sigma(M_1)$ other than $F$. See Figure \ref{f:AmlgSpine}. Hence, the spine of $\amlg{H}$ is obtained from $\Sigma(M_1) \cup \Sigma(M_2)$ by connecting with a vertical arc and removing a copy of $F$. The result thus follows. 

        \begin{figure}[htbp]
        \psfrag{X}{$M_1$}
        \psfrag{Y}{$M_2$}
        \psfrag{F}{$F$}
        \vspace{0 in}
        \begin{center}
       \includegraphics[width=3 in]{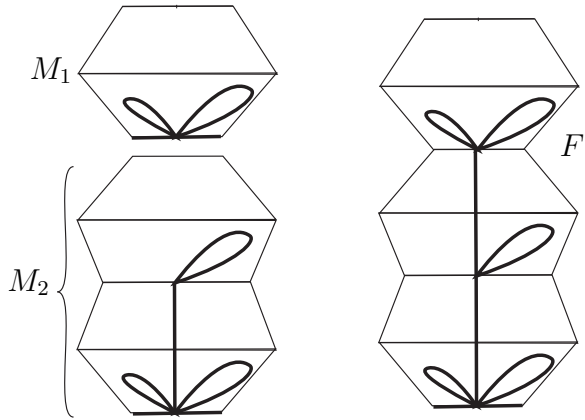}
       \caption{The spine of $\amlg{H}$ is obtained from $\Sigma(M_1) \cup \Sigma(M_2)$ by connecting with a vertical arc and removing a copy of $F$.}
        \label{f:AmlgSpine}
        \end{center}
        \end{figure}

We now move on to the case where $F$ is a connected, non-separating surface. Now $|F|-n+1=1-1+1=1$, so we need to establish \[\gen(H)=\gen(H(M_1))+\gen(H(M_2))-\gen(F)+1.\]
Let $N$ denote the manifold obtained from $M$ by cutting open along $F$. Let $\Sigma(N)$ denote the spine of $\amlg{H(N)}$. As in the separating case, the spine of $\amlg{H}$ is obtained from $\Sigma(N)$ by first removing a copy of $F$. This drops the genus by the genus of $F$. To complete the formation of the spine of $\amlg{H}$, we attach a vertical arc through the compression body incident to the other copy of $F$. As this arc connects what remains of $\Sigma(N)$ to itself, this increases the genus by one. 

To proceed from the case where $|F|=1$ to arbitrary values of $|F|$, simply note that $M$ can be successively built up from $\{M_i\}$ by attaching along one component of $F$ at a time. The result thus follows by an elementary induction argument.
\end{proof}

\begin{cor}
\label{c:GHSgenus}
Let $H$ be a GHS of $M$. Let $\thin{H}^\circ$ denote the subset of $\thin{H}$ consisting of those elements that lie in the interior of $M$. Then
\begin{eqnarray*}
\gen(H)&=&\sum \limits _{H_+ \in \thick{H}} \gen(H_+) -\sum \limits _{H_- \in \thin{H}^\circ} \gen(H_-) \\
&&+ |\thin{H}^\circ| - |\thick{H}|+1
\end{eqnarray*}
\end{cor}

\begin{proof}
Let $F$ be the union of all of the surfaces in $\thin{H}^\circ$, and apply Lemma \ref{l:GenusSum}. Note that there is one element of $\thick{H}$ in each component of the complement of $\thin{H}^\circ$. So the number of such components is precisely $|\thick{H}|$. 
\end{proof}

It should be noted that an alternative approach to the material in this section would be to first define the genus of a GHS to be that given by the formula in Corollary \ref{c:GHSgenus}. Lemma \ref{l:GenusSum} then follows from this definition fairly quickly. However, to prove equivalence to the definition given here, one would need an additional lemma that asserts that genus does not change under weak reductions that are not destabilizations.

\section{Sequences of GHSs}
\label{s:LastDefSection}

\begin{dfn}
A {\it Sequence Of GHSs} (SOG), $\{H^i\}$ of $M$ is a finite sequence such that for each $i$ either $H^i$ or $H^{i+1}$ is obtained from the other by a weak reduction.
\end{dfn}

\begin{dfn}
If $\bf H$ is a SOG and $k$ is such that $H^{k-1}$ and $H^{k+1}$ are obtained from $H^k$ by a weak reduction then we say the GHS $H^k$ is {\it maximal} in $\bf H$. 
\end{dfn}

It follows that maximal GHSs are larger than their immediate predecessor and immediate successor. 

Just as there are ways to make a GHS ``smaller", there are also ways to make a SOG ``smaller". These are called {\it SOG reductions}, and are explicitly defined in Section 8 of \cite{gordon}. If the first and last GHS of a SOG are strongly irreducible and there are no SOG reductions then the SOG is said to be {\it irreducible}. For our purposes, all we need to know about SOG reduction is that the maximal GHSs of the new SOG are obtained from the maximal GHSs of the old one by weak reduction, and the following lemma holds:

\begin{lem}
\label{l:maximalGHS}
{\rm (\cite{gordon}, Lemma 8.9)} Every maximal GHS of an irreducible SOG is critical. 
\end{lem}

\begin{dfn}
The {\it genus} of a SOG is the maximum among the genera of its GHSs.
\end{dfn}

\begin{lem}
\label{l:GenusGoesDown}
If a SOG $\Lambda$ is obtained from an SOG $\Gamma$ by a reduction then the genus of $\Gamma$ is at least the genus of $\Lambda$. 
\end{lem}

\begin{proof}
Since weak reduction can only decrease the genus of a GHS, the genus of a SOG is the maximum among the genera of its maximal GHSs. But if one SOG is obtained from another by a reduction, then its maximal GHSs are obtained from GHSs of the original by weak reductions. The result thus follows from Lemma \ref{l:AmalgGenus}.
\end{proof}

\section{Barrier surfaces}
\label{s:BarrierSurfaces}

We begin this section with a brief description of the complexity of a gluing map, as defined in \cite{barrier}. Let $M$ be a compact, irreducible, (possibly disconnected) 3-manifold with incompressible boundary, such that no component of $M$ is an $I$-bundle. Suppose boundary components $F_1$ and $F_2$ of $M$ are homeomorphic.  Let $M_\phi$ be the manifold obtained from $M$ by gluing these boundary components together by the map $\phi:F_1 \to F_2$. 

Let $Q$ denote a properly embedded (possibly disconnected) surface in $M$ of maximal Euler characteristic, which is both incompressible and $\bdy$-incompressible, and is incident to both $F_1$ and $F_2$.  Then we define the {\it distance} of $\phi$ to be the distance between the loops of $\phi(F_1 \cap Q)$ and $F_2 \cap Q$. When the genus of $F_2$ is at least two, then this distance is measured in the curve complex of $F_2$. If $F_2 \cong T^2$, then this distance is measured in the Farey graph. 

We are now prepared to state the main result of \cite{barrier}. 

\begin{thm}
\label{t:Barrier}
\cite{barrier}
Let $F$ denote the image of $F_1$ in $M_\phi$. There is a constant $K$, depending linearly on $\chi(Q)$, such that if the distance of $\phi \ge Kg$, then any incompressible, strongly irreducible, or critical surface $H$ in $M_\phi$ of genus at most $g$ can be isotoped to be disjoint from $F$. 
\end{thm}

The theorem given in \cite{barrier} is quite a bit stronger than this. There we prove that the conclusion holds for all {\it topologically minimal} surfaces. As incompressible, strongly irreducible, and critical surfaces are examples of topologically minimal surface, the version of the theorem stated above follows. 

Theorem \ref{t:Barrier} motivates us to make the following definition:

\begin{dfn}
An incompressible surface $F$ in a 3-manifold $M$ is a {\it $g$-barrier surface} if any incompressible, strongly irreducible, or critical surface in $M$ whose genus is at most $g$ can be isotoped to be disjoint from $F$. 
\end{dfn}

By employing Theorem \ref{t:Barrier} we may construct 3-manifolds with any number of $g$-barrier surfaces. Simply begin with a collection of 3-manifolds and successively glue boundary components together by ``sufficiently complicated" maps.

\begin{lem}
\label{l:FparallelToThinLevelGHS}
Let $M$ be a 3-manifold which has a $g$-barrier surface $F$. Let $H$ be a genus $g$ strongly irreducible or critical GHS of $M$. Then $F$ is isotopic to a thin level of $H$.
\end{lem}

\begin{proof}
Since the genus of $H$ is $g$, it follows from Corollary \ref{c:GHSgenus} that the genus of every thick and thin level of $H$ is at most $g$. By Theorem \ref{t:IncompressibleThinLevels} we know that each thin level of $H$ is incompressible. Since $F$ is a $g$-barrier surface, it can be isotoped to be disjoint from every thin level. But then $F$ is contained in $M(H_+)$, for some thick level, $H_+$. The surface $H_+$ is either strongly irreducible or critical, so again since $F$ is a $g$-barrier surface it may be isotoped to be disjoint from $H_+$. The surface $F$ can thus be isotoped into a compression body, $\CC$.  But every incompressible surface in $\CC$ is parallel to some component of $\bdy _-\CC$. Each such component is a thin level of $H$. 
\end{proof}




\begin{lem}
\label{l:FparallelToThinLevelSOG}
Let $M$ be a 3-manifold which has a $g$-barrier surface $F$. Let $\bf H$ be a genus $g$ irreducible SOG of $M$. Then $F$ is isotopic to a thin level of every element of $\bf H$. 
\end{lem}

\begin{proof}
By Lemma \ref{l:maximalGHS} each maximal GHS of $\bf H$ is critical. Hence, by Lemma \ref{l:FparallelToThinLevelGHS} $F$ is isotopic  to a thin level of every maximal GHS of $\bf H$. But every other GHS of $\bf H$ is obtained from a maximal GHS by a sequence of weak reductions and destabilizations. Such moves may create new thin levels, but will never destroy an incompressible thin level. Hence, $F$ is isotopic to a thin level of every element of $\bf H$. 
\end{proof}

\section{Lower bounds on stabilizations.}
\label{s:CounterExamples}

\begin{lem}
\label{l:LowerBoundTheorem}
Let $\{F_i\}_{i=1}^n$ denote a collection of $g$-barrier surfaces in $M$. Let $\{M_k\}_{k=1}^m$ denote the closures of the components of $M-\cup F_i$. Let ${\bf H}=\{H^j\}$ denote an irreducible SOG of $M$. If $F_1$ is isotopic to a unique thin level of $H^1$ and $H^m$, but is oriented in opposite ways in each of these GHSs, then \[\gen({\bf H}) \ge \min \{g, \sum \limits _{k} \gen(M_k)-\sum \limits _{i \ne 1} \gen(F_i)+n-m+1\}.\]
\end{lem}

\begin{proof}
Assume $\gen({\bf H}) \le g$. By Lemma \ref{l:FparallelToThinLevelSOG} the surface $F_1$ is then isotopic to a thin level of every GHS of $\bf H$. Weak reduction can not simultaneously kill one thin level and create a new one, so it follows that for some $j$, there is a GHS $H^j$ of $\bf H$ where $F_1$ is isotopic to two thin levels, but oriented differently. Let $P$ denote a submanifold of $M$ cobounded by two such thin levels. Let $\overline{P}=\amlg{H^j(P)}$. Then $\overline{P}$ is a Heegaard splitting of $P$ that does not separate its boundary components. As $P$ is homeomorphic to $F_1 \times I$, it follows from \cite{st:93} that $\overline{P}$  is a stabilization of two copies of $F_1$, connected by a tube. Hence, $\gen(\overline{P}) \ge 2 \gen(F_1)$. 

By Lemma \ref{l:FparallelToThinLevelSOG}, for each $i$ there is a thin level of $H^j$ which is isotopic to $F_i$. For each $i \ne 1$ choose one such thin level, and call it $F_i^j$. If we cut $M$ along $\{F^j_i|i \ne 1\}$, and then remove the interior of $P$, we obtain a collection of manifolds homeomorphic to $\{M_k\}$. We denote this collection as $\{M^j_k\}$. For each $k$, let $\overline{M_k}=\amlg{H^j(M^j_k)}$. It thus follows from Lemma \ref{l:GenusSum} that

\begin{eqnarray*}
\gen({\bf H}) & \ge & \gen(H^j)\\
&=& \sum \limits _k \gen(\overline{M_k})-\sum \limits _{i \ne 1} \gen(F_i) +\gen(\overline{P}) - 2\gen(F_1)\\
&&\hspace{.5in} +(n+1)-(m+1)+1\\
&\ge &\sum \limits _k \gen(M_k)-\sum \limits _{i \ne 1} \gen(F_i)+n-m+1
\end{eqnarray*}
\end{proof}

In the next three theorems we present our counter-examples to the Stabilization Conjecture. 

\begin{thm}
\label{t:FlipCounterExample}
For each $n \ge 4$ there is a closed, orientable 3-manifold that has a genus $n$ Heegaard splitting which must be stabilized at least $n-2$ times to become equivalent to the splitting obtained from it by reversing its orientation. 
\end{thm}

\begin{proof}
Let $M_1$ and $M_2$ be 3-manifolds that have one boundary component homeomorphic to a genus $g$ surface, $F$ (where $g \ge 2$). For each $i$ the manifold $M_i$ has a strongly irreducible Heegaard splitting $H_i$ of genus $g+1$.

Now glue $M_1$ and $M_2$ along their boundaries by a ``sufficiently complicated" map, so that by Theorem \ref{t:Barrier} the gluing surface $F$ becomes a $(2g+2)$-barrier surface. Let $M$ be the resulting 3-manifold. A GHS $H^1$ of $M$ is then defined by:
	\begin{enumerate}
		\item $\thick{H^1}=\{H_1, H_2\}$
		\item $\thin{H^1}=\{F\}$
	\end{enumerate}
	
Choose an orientation on $H^1$. Let $H^*$ denote the GHS with the same thick and thin levels, but with opposite orientation. Then $\amlg{H^*}$ is a Heegaard splitting of $M$ that is obtained from the splitting  $\amlg{H^1}$ by reversing its orientation. By Corollary \ref{c:GHSgenus} the genera of these splittings is \[n=2(g+1)-g=g+2.\] We now claim that these splittings are not equivalent after any less than $g=n-2$ stabilizations. Let $H$ denote the minimal genus common stabilization of these splittings. We must show $\gen(H) \ge (g+2)+g=2g+2$. 

Let ${\bf H}=\{H^i\}_{i=1}^n$ be the SOG where 
	\begin{enumerate}
		\item $H^1$ is as defined above, 
		\item $H^n=H^*$,
		\item for some $1<j<n$, $\thick{H^j}=\{H\}$ and $\thin{H^j}=\emptyset$, and
		\item $H^j$ is the only maximal GHS in $\bf H$. 
	\end{enumerate}

Let ${\bf K}$ be a SOG obtained from $\bf H$ by a maximal sequence of SOG reductions. By Lemma \ref{l:GenusGoesDown}, $\gen({\bf H}) \ge \gen({\bf K})$. Since the orientations on $F$ disagree in the initial and final GHS of $\bf H$, this must also be true of $\bf K$. Hence, by Lemma \ref{l:LowerBoundTheorem}, 
\begin{eqnarray*}
\gen({\bf K}) & \ge & \gen(M_1) + \gen(M_2)\\
&=&(g+1)+(g+1)\\
&=&2g+2
\end{eqnarray*}
Hence, $\gen(H) = \gen({\bf H}) \ge \gen({\bf K}) \ge 2g+2$. 
\end{proof}

\begin{thm}
\label{t:TorusBoundaryCounterExamples}
For each $n \ge 5$ there is an orientable 3-manifold whose boundary is a torus, that has two genus $n$ Heegaard splittings which must be stabilized at least $n-4$ times to become equivalent.
\end{thm}

\begin{proof}
Let $M_1$ and $M_2$ be 3-manifolds that have one boundary component homeomorphic to a genus $g$ surface, $F$ (where $g \ge 2$). The manifold $M_1$ also has a boundary component $T$ that is a torus. The manifold $M_2$ has no boundary components other than $F$. For each $i$ the manifold $M_i$ has a strongly irreducible Heegaard splitting $H_i$ of genus $g+1$. The manifold $M_1$ then has a genus $g+2$ Heegaard splitting $G_1$, obtained from $H_1$ by boundary stabilizing along $T$. 

Now glue $M_1$ and $M_2$ along their genus $g$ boundary components by a ``sufficiently complicated" map, so that by Theorem \ref{t:Barrier} the gluing surface $F$ becomes a $(2g+2)$-barrier surface. Let $M$ be the resulting 3-manifold. GHSs $H^1$ and $H^*$  of $M$ are then defined by:
	\begin{enumerate}
		\item $\thick{H^1}=\{H_1, H_2\}$
		\item $\thick{H^*}=\{G_1,H_2\}$
		\item $\thin{H^1}=\thin{H^*}=\{F, T\}$
	\end{enumerate}
	
Choose orientations on $H^1$ and $H^*$ so that the orientations on $T$ agree. Then both $\amlg{H^1}$ and $\amlg{H^*}$ are Heegaard splittings of $M$, with $T$ on the same side of each.  Hence, these two splittings have some common stabilization, $H$. Note also that the orientations on $F$ in $H^1$ and $H^*$ necessarily disagree. See Figure \ref{f:H1H*Orientations}.

        \begin{figure}[htbp]
        \psfrag{1}{$H^1$}
        \psfrag{2}{$H^*$}
        \psfrag{F}{$F$}
        \psfrag{H}{$H_1$}
        \psfrag{G}{$G_1$}
        \psfrag{h}{$H_2$}
        \psfrag{T}{$T$}
        \vspace{0 in}
        \begin{center}
       \includegraphics[width=4.5 in]{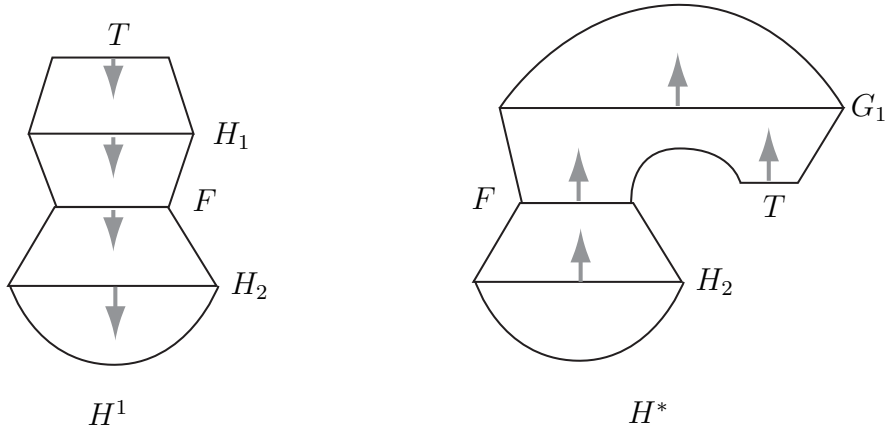}
       \caption{If the orientations on $T$ in $H^1$ and $H^*$ agree, then the orientations on $F$ disagree.}
        \label{f:H1H*Orientations}
        \end{center}
        \end{figure}

By Corollary \ref{c:GHSgenus} the genus of $\amlg{H^1}$  is \[2(g+1)-g=g+2.\] The genus of $\amlg{H^*}$ is one higher, $g+3$. Let this number be $n$. We now claim that we must stabilize $\amlg{H^*}$ at least $g-1=n-4$ times to obtain a stabilization of $\amlg{H^1}$. In other words, we claim \[\gen(H) \ge (g+3)+(g-1)=2g+2.\] 

Let ${\bf H}=\{H^i\}_{i=1}^n$ be the SOG where 
	\begin{enumerate}
		\item $H^1$ is as defined above, 
		\item $H^n=H^*$,
		\item for some $1<j<n$, $\thick{H^j}=\{H\}$ and $\thin{H^j}=\emptyset$, and
		\item $H^j$ is the only maximal GHS in $\bf H$. 
	\end{enumerate}

Let ${\bf K}$ be a SOG obtained from $\bf H$ by a maximal sequence of SOG reductions. By Lemma \ref{l:GenusGoesDown}, $\gen({\bf H}) \ge \gen({\bf K})$. Since the orientations on $F$ disagree in the initial and final GHS of $\bf H$, this must also be true of $\bf K$. Hence, by Lemma \ref{l:LowerBoundTheorem}, 
\begin{eqnarray*}
\gen({\bf K}) & \ge & \gen(M_1) + \gen(M_2)\\
&=&(g+1)+(g+1)\\
&=&2g+2
\end{eqnarray*}
Hence, $\gen(H) = \gen({\bf H}) \ge \gen({\bf K}) \ge 2g+2$. 
\end{proof}

\begin{thm}
\label{t:ClosedCounterExamples}
For each $n \ge 8$ there is a closed, orientable 3-manifold that has a pair of genus $n$ Heegaard splittings which must be stabilized at least $\frac{1}{2}n -3$ times to become equivalent (regardless of their orientations).
\end{thm}

\begin{proof}
Let $M_1$, $M_2$, $M_3$, and $M_4$ be 3-manifolds as follows. Each of these manifolds has one boundary component homeomorphic to a genus $g$ surface, $F$ (where $g \ge 2$), and a Heegaard splitting $H_i$ of genus $g+1$ that separates $F$ from any other boundary component. The manifolds $M_1$ and $M_2$ have a second boundary component, which is a torus. The manifold $M_3$ has two toroidal boundary components. The manifold $M_4$ has no boundary components other than $F$. For $i=1$ and $2$ the manifolds $M_i$  also have a second Heegaard surface, $G_i$, of genus $g+2$ obtained from $H_i$ by boundary stabilizing along the torus boundary component. 

Now glue all four manifolds together as in Figure \ref{f:H1GHS} by ``sufficiently complicated" maps so that by Theorem \ref{t:Barrier} both copies of $F$, and both gluing tori, become $(3g+3)$-barrier surfaces. Let $M$ be the resulting 3-manifold. For $i=1$ and 2 let $T_i$ denote the torus between $M_i$ and $M_3$. Let $F_1$ denote the copy of $F$ between $M_1$ and $M_2$, and $F_2$ the copy of $F$ between $M_3$ and $M_4$. 

        \begin{figure}[htbp]
        \psfrag{a}{$M_1$}
        \psfrag{b}{$M_2$}
        \psfrag{c}{$M_3$}
        \psfrag{d}{$M_4$}
        \psfrag{G}{$G_1$}
        \psfrag{g}{$G_2$}
        \psfrag{t}{$T_1$}
        \psfrag{t}{$T_1$}
        \psfrag{T}{$T_2$}
        \psfrag{f}{$F_1$}
        \psfrag{F}{$F_2$}
        \psfrag{1}{$H_1$}
        \psfrag{2}{$H_2$}
        \psfrag{3}{$H_3$}
        \psfrag{4}{$H_4$}
        \vspace{0 in}
        \begin{center}
       \includegraphics[width=4.5 in]{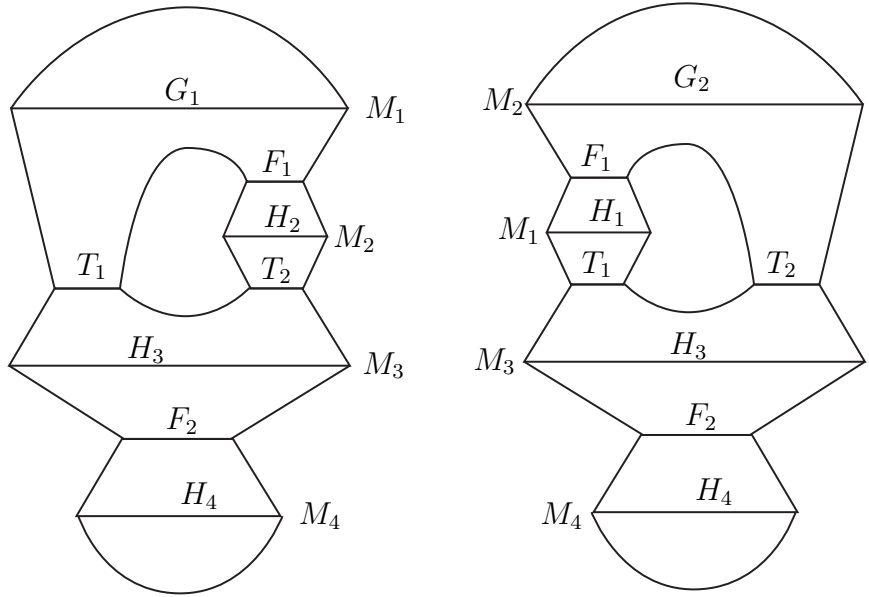}
       \caption{The GHSs, $H^1$ and $H^*$.}
        \label{f:H1GHS}
        \end{center}
        \end{figure}

We now define two GHSs $H^1$ and $H^*$ of $M$ (See Figure \ref{f:H1GHS}):

\begin{enumerate}
	\item $\thick{H^1}=\{G_1, H_2, H_3, H_4\}$
	\item $\thick{H^*}=\{H_1, G_2, H_3, H_4\}$. 
	\item $\thin{H^1}=\thin{H^*}=\{F_1, F_2, T_1, T_2\}$
\end{enumerate}

By definition, $\amlg{H^1}$ and $\amlg{H^*}$ are both Heegaard splittings of $M$. By Corollary \ref{c:GHSgenus} the genera of these splittings is \[n=3(g+1)+(g+2)-2g-2+1=2g+4.\] We claim that no matter what orientation is chosen for these GHSs, they are not equivalent after any less than $g-1=\frac{1}{2}n-3$ stabilizations. Let $H$ denote the minimal genus common stabilization of these splittings. We must show $\gen(H) \ge (2g+4)+(g-1)=3g+3$. 

Orient $H^1$ and $H^*$. Note that if these orientations agree on $F_1$ then they disagree on $F_2$. See Figure \ref{f:F1F2Orientations}. Hence, any SOG that interpolates between $H^1$ and $H^*$ must reverse the orientation of either $F_1$ or $F_2$.

        \begin{figure}[htbp]
        \psfrag{a}{(a)}
        \psfrag{b}{(b)}
        \psfrag{1}{$H^1$}
        \psfrag{f}{$F_1$}
        \psfrag{F}{$F_2$}
        \vspace{0 in}
        \begin{center}
       \includegraphics[width=4.5 in]{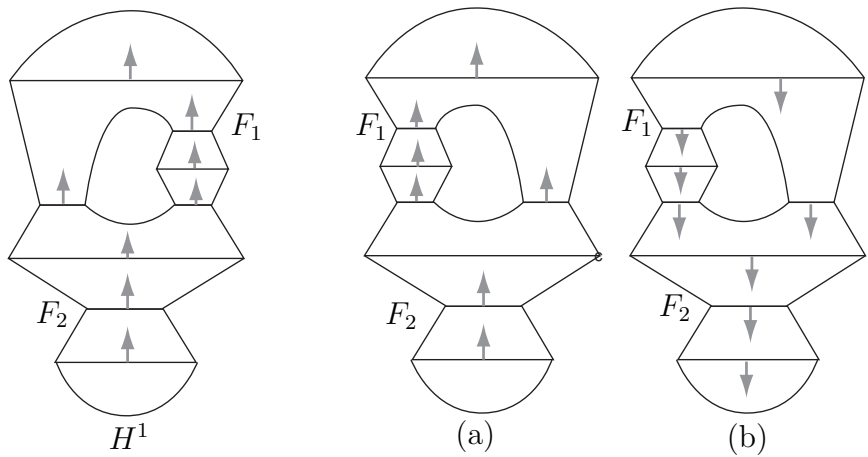}
       \caption{An orientation on $H^1$ and two possible orientations on $H^*$. In $H^1$ the manifold $M_1$ is above $F_1$. In Case (a) the manifold $M_1$ is below $F_1$. Hence, the orientations on $F_1$ in $H^1$ and $H^*$ disagree. In Case (b) the orientations on $F_2$ disagree.}
        \label{f:F1F2Orientations}
        \end{center}
        \end{figure}

Let ${\bf H}=\{H^i\}_{i=1}^n$ be the SOG where 
	\begin{enumerate}
		\item $H^1$ is as defined above, 
		\item $H^n=H^*$,
		\item for some $1<j<n$, $\thick{H^j}=\{H\}$ and $\thin{H^j}=\emptyset$, and
		\item $H^j$ is the only maximal GHS in $\bf H$. 
	\end{enumerate}

Let ${\bf K}=\{K^i\}$ be a SOG obtained from $\{H^i\}$ by a maximal sequence of SOG reductions. By Lemma \ref{l:GenusGoesDown}, $\gen({\bf H}) \ge \gen({\bf K})$. By Lemma \ref{l:LowerBoundTheorem}, 
\begin{eqnarray*}
\gen({\bf K}) & \ge & \sum \limits_{i=1} ^4 \gen(M_i) - \gen(T_1)-\gen(T_2)-\gen(F)+1\\
&=&4(g+1)-2-g+1\\
&=&3g+3
\end{eqnarray*}
Hence, $\gen(H) = \gen({\bf H}) \ge \gen({\bf K}) \ge 3g+3$. 
\end{proof}

\bibliographystyle{alpha}

\end{document}